\documentclass[11pt,reqno]{amsart}

\usepackage{amsmath, amsfonts, amsthm, amssymb, color}

\newtheorem{Theorem}{Theorem}[section]
\newtheorem{Lemma}{Lemma}[section]
\newtheorem{Proposition}{Proposition}[section]

\theoremstyle{definition}
\newtheorem{Definition}{Definition}[section]

\theoremstyle{remark}
\newtheorem{Remark}{Remark}[section]

\numberwithin{equation}{section}

\renewcommand{\u}{{\bf u}}

\newcommand{\R}{{\mathbb R}}
\newcommand{\Dv}{{\rm div}}

\newcommand{\x}{{\bf x}}

\newcommand{\m}{{\bf m}}

\def\f{\frac}
\renewcommand{\O}{\Omega}

\def\D{\Delta }

\def\hf1{^\f{1}{1-\xi^2}}

\def\be{\begin{equation}}
\def\en{\end{equation}}
\def\bs{\begin{split}}
\def\es{\end{split}}

\renewcommand{\v}{{\bf v}}

\newcommand{\supp}{{\rm supp}}

\newcommand{\eps}{{\varepsilon}}

\author{Irene M. Gamba}
\address{Department of Mathematics, The University of Texas at Austin,
                           Austin, Texas 78712.}
\email{gamba@math.utexas.edu}
\author{Cheng Yu}
\address{Department of Mathematics, The University of Texas at Austin,
                           Austin, Texas 78712.}
\email{yucheng@math.utexas.edu}

\title
[Navier-Stokes-Vlasov-Boltzmann equations]
{Global weak solutions to compressible Navier-Stokes-Vlasov-Boltzmann systems for spray dynamics}


\date{\today}

\begin{document}

\begin{abstract}
This work concerns the global existence of the weak solutions to a system of partial differential equations modeling the evolution of particles in the fluid. That system is given by a coupling between the standard isentropic compressible Navier-Stokes equations for the macroscopic description of a gas fluid flow, and a Vlasov-Boltzmann type equation governing the evolution of  spray droplets modeled as  particles with varying radius.
    We establish the existence of global weak solutions with finite energy, whose density of gas satisfies the renormalized mass equation.
   The proof, is partially motivated by the work of Feireisl- Novotny-Petzeltov \cite{FNP}  on the weak solutions of the compressible Navier-Stokes equations      coupled to the kinetic problem for the spray droplets extending the techniques of  Legger and Vasseur \cite{LV} developed for the incompressible fluid-kinetic  system.

\end{abstract}

\maketitle

\section{Introduction}
A large variety models  describing sprays dynamics, introduced by Williams \cite{W}, are obtained by coupling a  of fluid mechanics equation and a kinetic one describing the spray as perfect bubbles. In such a system models, the gas surrounding the spray  is described by classical fluid macroscopic quantities: its density $\rho(t,x)\geq 0$ and velocity $\u(t,x)$. Depending on the physical properties of such gas fluid, the evolution of those quantities is ruled by the Navier-Stokes or Euler Equations compressible flows.  Fluid viscosity because an important physical quantity and it is model in the classical compressible Navier Stokes framework.

The spray droplet evolution is assumed to be given by independent distributed continuum randon variables described by a distribution function $f=f(t,x,\v,r)\geq 0$ given the probability of finding a droplet  with center at position $x$, with radious $r$, time $t$, moving with  velocity $\v$. Depending on physical properties of the droplets, the evolution of $f$ is governed by a kinetic equation given by a Vlasov-linear Boltzmann  model, were the non-local Boltzmann operator models collisions and breakup.


In such a system models, the coupling comes from drag force in the fluid equation and the acceleration in the Vlasov term of kinetic equation, as the fluid a dense phase and  the droplets in a disperse phase strongly  interact on each other.

More specifically we consider an spray model given by  the following  Navier-Stokes-Vlasov-Boltzmann system of equations for droplet particles dispersed in a compressible viscous fluid
\begin{equation}\label{NS1}
\rho_{t}+\Dv{(\rho\u)}=0,
\end{equation}
\begin{equation}\label{NS2}
(\rho\u)_{t}+\Dv(\rho\u\otimes \u)+\nabla p-\mu\D\u-\lambda\nabla\Dv\u=\mathbf{F}_{r}(t,x),
\end{equation}
\begin{equation}\label{VB}
f_{t}+\xi\cdot\nabla_{x}f+\Dv_{\xi}{(F f)}=Q(f),
\end{equation}
for $(x,\xi,r, t)$ in $\O\times\R^3\times[a,b]\times[0,\infty)$,
where $\O\subset\R^3$,
$\rho$ is the density of the fluid,
$\u$ is the velocity of the fluid, $p=\rho^{\gamma}$ is the pressure for some $\gamma>1$, The viscosity coefficients  $\mu$ and $\lambda$ have a relationship
$$\mu>0,\quad\lambda+\frac{\mu}{3}\geq 0.$$
The probability density distribution function $f(x,\xi,r,t)$ of gas particles
depends on the physical position $x\in\O$, the velocity of particle $\xi\in\R^3$, the radius of a particle $r\in [a,b],$ and  the time $t\in[0,T]$,
 where  $a,b>0$ are the constants.
 The zero moment of the gas particles density is
\begin{equation}
\label{zero moment}
\mathfrak{n}(t,x)=\int_a^b\int_{\R^3}r f\,d\xi\,dr,
\end{equation}
and the kinetic  current (first moment) is
\begin{equation}
\label{first moment}
 j(t,x)=\int_a^b\int_{\R^3}r \xi f\,d\xi\,dr.
\end{equation}

The interaction of the fluid and particles is through the
drag force exerted by the fluid onto the particles.
In \eqref{VB}, $F$ stands for the acceleration felt by the droplets. It was
typically given by the following formula which is known as Stokes' law,
\begin{equation}
\label{force}
F(x,\xi,r,t)=\frac{9\mu}{2\rho_l}\frac{\u-\xi}{r^2},
\end{equation}
where $\rho_l$ is the constant density of the liquid, $\mu$ is the dynamic viscosity.
 In \eqref{NS2}, the right hand side term \begin{equation}
 \label{the drag force}
 \mathbf{F}_r(t,x)=-\int_a^{b}\int_{\R^3}\frac{4}{3}\rho_l r^3 f F\,d\xi\,dr.
 \end{equation}
 The operator $Q(f)$ is taking into account the complex phenomena happening at the level of the droplet particles, such as collisions and breakup.
 Assuming that the droplets keep the same velocities before and after breakup,  the  operator could be obtained
\begin{equation}
\label{operator}
Q(f)(x,\xi,r,t)=-\nu f(x,\xi,r,t)+\nu\int_{r>r^*}B(r^*,r) f(x,\xi,r^*,t)\,d r^*,
\end{equation}
where $\nu\geq 0$ is the fragmentation rate, and $B=B(r^*,r)\geq 0$ is related to the probability of ending up with droplets particles of radius $r$ out of the breakup of droplets particles of radius $r^*$.
This is a typical form of the breakage model kernel.

Without loss of generality we take $\rho_l=\frac{9\mu}{2}$ throughout the paper.
The fluid-particle system \eqref{NS1}-\eqref{operator} arises in many applications such as
sprays, aerosols, and more general  two phase flows where
one phase (disperse) can be considered as a suspension of particles onto the other
one (dense) regarded as a fluid.
System \eqref{NS1}-\eqref{operator} or its variants have been used in sedimentation
of solid grain by external forces, for fuel-droplets in combustion theory (such as in the study of engines),  chemical engineering,  bio-sprays in medicine, waste water treatment, and pollutants in the air.
We refer the readers to \cite{BD,CP,D,DR,DV,GHMZ,O,RM,W} for more physical backgrounds, applications and discussions of  the fluid-particle systems.
Leger-Vasseur have shown the existence of global weak solutions of the incompressible version of Vlasov-Boltzmann-Navier-Stokes equations.
\\
\newline
The aim of this current paper is to establish the existence of global weak solutions to
the system \eqref{NS1}-\eqref{operator}, or equivalently to
\begin{equation}\label{NSVB1}
\rho_{t}+\Dv{(\rho\u)}=0,
\end{equation}
\begin{equation}\label{NSVB2}
(\rho\u)_{t}+\Dv(\rho\u\otimes \u)+\nabla p-\mu\D\u-\lambda\nabla\Dv\u=-\int_a^b\int_{\R^3}r(\u-\xi)f\,d\xi\,dr,
\end{equation}
\begin{equation}\label{NSVB3}
f_{t}+\xi\cdot\nabla_{x}f+\Dv_{\xi}{\left(\frac{(\u-\xi)f}{r^2}\right)}=Q(f),
\end{equation}
 subject to the following initial data:
\begin{equation}\label{intial data}
\rho|_{t=0}=\rho_{0}(\x)\ge 0,\;\;
(\rho\u)|_{t=0}=\m_{0}(x),\;\;
f|_{t=0}=f_{0}(x,\xi,r),
\end{equation}
where $Q(f)$ is given by \eqref{operator}. The collision operator $Q(f)$ satisfies the following \textbf{hypotheses A}:\\
 \uppercase\expandafter{\romannumeral1}.  $B\in C^1(\R^{+}\times\R^{+})$, $B\geq 0,$ and $B(r,r^{*})=0$ if $r\geq r^*.$\\
for all $(r,r^*)\in \R^{+}\times\R^{+}.$\\
 \uppercase\expandafter{\romannumeral2}. $\int_a^bB(r,r^*)\,dr=\int_{R(a)}^{R(b)}B(r,r^*)\,dr$, with
$$R(r)=\sqrt[3]{r^{*^{3}}-r^{3}}\quad\text{ and }\quad 0\leq a\leq b\leq \frac{r^*}{\sqrt[3]{2}}.$$
 \uppercase\expandafter{\romannumeral3}. $\int_0^{\frac{r^*}{\sqrt[3]{2}}}B(r,r^*)\,dr=\int_{\frac{r^*}{\sqrt[3]{2}}}^{r^*}B(r,r^*)\,dr=1,$ which without loss of generality, both integrals to be one by renormalization.
\vskip0.3cm

Our  strategy  to solve the initial value problem for system \eqref{NSVB1}-\eqref{intial data} with assumptions (I-III) given above, consists in combining recent known techniques by the Feireisl-Novotn\'{y}-Petzeltov\'{a} \cite{FNP} by the use of their
 regularization method for solving the fluid system using the compressible  Navier-Stokes system,    in an iteration that couples such fluid equation to the initial value problem of the Vlasov-linear Boltzmann for the droplet particle evolution. For this coupling we adapt  a recent approach proposed by Legger and Vasseur \cite{LV} where they solved the same kinetic equation  coupled to a fluid given by the incompressible Navier-Stokes system.

\medskip

The manuscript is organized as follows. In the next Section 2 we introduce some fundamentals and prove, for a fixed droplet particle distribution $f(x,\xi,r,t)$,  the basic a priori momentum and energy identities for the compressible Navier Stokes' equation.

In section 3,  we introduce first the two level $\eps, \delta$-regularization technique  from \cite{FNP}
 to system \eqref{NSVB1}-\eqref{intial data} by adding as $\eps$-viscous term to  mass equation and an    $\eps$-modification of the momentum equation that preserves the energy identities   for fix $f(x,\xi,r,t)$, derived in section 2, and a $\delta$-modification that modify the pressure law.
 In addition, we employ the techniques from    \cite{FNP}, where
 each $\eps,\delta$-regularized  Navier Stokes   (\ref{NSVB1}-\ref{NSVB2})  part  is solved uniquely by  a  $k$-finite dimensional approximating model, introduced in \cite{FNP} and \cite{F04}. Then   for each $u^{\eps,\delta}_{\bf k}$,  we finally  the Vlasov-linear-Boltzmann equation \eqref{NSVB3}  using the approach of \cite{LV}, whole solution is an approximating $f^{\eps,\delta}_{\bf k}$.
  This iteration is shown to construct  unique solutions $(\rho^{\eps,\delta}_{\bf k}, u^{\eps,\delta}_{\bf k}, f^{\eps,\delta}_{\bf k})$   to the   $\eps,\delta, {\bf k}$-approximating system to (\ref{NSVB1}-\ref{NSVB2}-\ref{NSVB3}) by means of a fixed point argument in a Banach space, where  initial data is modified by  introducing the parameter $\underline{\rho}>0$ that keep our  the $\rho^{\eps,\delta}_{\bf k}$ estimates  bounded below from vacuum uniformly in $\eps,\delta$ and  ${\bf k}$.
In addition,  we show that the unique solutions  $(\rho^{\eps,\delta}_{\bf k}, u^{\eps,\delta}_{\bf k}, f^{\eps,\delta}_{\bf k})$ for the $\eps,\delta, {\bf k}$-approximating system, satisfy  momentum and energy identities, uniformly in $\eps, \delta$ and $ {\bf k}$,    and  the approximating density $\rho^{\eps,\delta}_{\bf k}$ is bounded below by $\underline{\rho}>0$ uniformly in $\eps$ and ${\bf k}$.

Finally,   in Section 4,   we study the limiting process to obtain a global weak solution to (\ref{NSVB1}-\ref{NSVB2}-\ref{NSVB3}),  by  first performing the limit  $k\to \infty$,  next the limit  $\eps\to 0$,  and last the  limit $\delta \to 0$ obtaining a  limiting triple $(\rho,u,f)$ whose initial data has $\rho(x,0)\ge \underline{\rho}>0$ for an arbitrary  $\underline{\rho}>0$. So the existence of solution in then proved for any initial data who density $\rho$ may vanish locally.


\bigskip

\section{A Priori Estimates}

In this section, we derive some fundamental a priori estimates for each equation on the system \eqref{NSVB1}-\eqref{NSVB3}. They are crucial to show the existence of weak solutions upon passing to the limits in the regularized approximation scheme.

We first recall the notation of renormalized solutions, \cite{L, FNP,F04}. In fact, multiplying \eqref{NSVB1} by $b'(\rho)$ we deduce
\begin{equation}
\label{renomalized solution}
h(\rho)_t+\Dv(h(\rho)\u)+(h'(\rho)\rho-h(\rho))\Dv\u=0
\end{equation}
for any differentiable function $h$. Thus, we give the following definition.
\begin{Definition}
\label{definition of renoamalized solution}
Equation \eqref{NSVB1} is satisfied in the renormalized sense, more specifically,
equation \eqref{renomalized solution} holds in the distributional sense,
for any $h\in C^1(\R)$ such that
$$h'(z)=0\quad\text{ for all } |z|\geq M,$$
for some constant $M>0$.
\end{Definition}

Here, for the sake of simplicity we will consider the case of bounded domain with periodic boundary conditions, namely $\O=\mathbb{T}^3$.
In this paper, we assume that
\begin{equation}\label{2.4}
 \left\{
   \begin{array}{c}
\rho_{0}\geq 0 \; \text{ almost everywhere in } \O,\;
 \m_{0}\in L^2(\O),\\\quad \m_{0}=0\; \text{ almost everywhere on } \{\rho_{0}=0\},
\quad \frac{|\m_{0}|^2}{\rho_{0}} \in L^1(\O),
\\
 f_{0}\in L^{\infty}\cap L^1(\O\times\R^3\times\R^+),\quad r^3|\xi|^{3}f_{0} \in L^1(\O\times\R^3\times\R^+).
\end{array}
  \right.\end{equation}

\begin{Definition} \label{D1} The triple
 $(\rho,\u,f)$ is a global weak solution to problem \eqref{NSVB1}-\eqref{2.4} if,  for any $T>0$, the following properties hold,\\
 \romannumeral1. $\rho\geq 0, \quad \rho \in C([0,T];L^{\gamma}(\O)),$ $\u \in L^{2}(0,T;H^{1}_0(\O)),$ $\rho|\u|^{2} \in L^{\infty}(0,T;L^{1}(\O));$\\
 \romannumeral2. $f(t,x,\xi,r)\geq 0, \text { for any } (t,x,\xi,r ) \in (0,T)\times \O\times \R^3\times\R^+;$\\
 \romannumeral3.  $f \in L^{\infty}(0,T;L^{\infty}(\O\times\R^3\times\R^+) \cap L^1(\O\times\R^3\times\R^+));$\\
\romannumeral4.  $r^3|\xi|^3f \in L^{\infty}(0,T;L^1(\O\times\R^3\times\R^+));$\\
\romannumeral5. Equation \eqref{NSVB1} is satisfied in the renormalized sense.\\
\romannumeral6. For any $\varphi \in C^{1}([0,T]\times \O)$, for almost everywhere $t$, the following identify holds
\begin{equation}
\label{weak-1}
\begin{split}
&\quad\quad-\int_{\O}\m_{0}\cdot\varphi(0,x)\;dx+\int_{0}^{t}\!\!\int_{\O}\bigg(-\rho\u\cdot\partial_{t}\varphi
-(\rho\u\otimes\u):\nabla\varphi-\rho^{\gamma}\nabla\varphi
\\&\qquad\qquad\qquad\qquad
+\mu \nabla\u\cdot\nabla\varphi+\lambda\Dv\u\Dv\varphi+\int_{\R^3}r f(\u-\xi)\cdot \varphi \,d\xi\,d r\bigg)\;dxdt=0;
\end{split}
\end{equation}
\romannumeral7. For any $\phi \in C^1([0,T]\times\O\times\R^3\times\R^+)$ with compact support with respect to $x$, $\xi$, and $r$,  such that $\phi(T,\cdot,\cdot,\cdot)=0$, the following identify holds
\begin{equation}
\label{weak-2}
\begin{split}
&\quad-\int_{0}^{T}\!\!\!\int_{\O}\int_{\R^3}f\left({\partial_t\phi+\xi\cdot\nabla_{x}\phi+\frac{(\u-\xi)}{r^2} \cdot\nabla_{\xi}\phi}\right)\;dxd\xi ds
\\&\quad\quad\quad\quad\quad=\int_{\O}\int_{\R^3}f_{0}\phi(0,\cdot,\cdot)\;dxd\xi+\int_0^T\int_{\O}Q(f) \phi\,dx\,dt;
\end{split}
\end{equation}
\romannumeral8. The energy inequality
\begin{equation}
\label{2.3+}
\begin{split}
&\int_{\O}\rho|\u|^{2}dx+\int_{\O}\int_{\R^3}f(1+|\xi|^{2})\;d\xi dx
 +2\mu\int_{0}^{T}\!\!\!\int_{\O}|\nabla\u|^2\;dxdt+2\lambda\int_{0}^{T}\!\!\!\int_{\O}|\Dv\u|^2\;dxdt
\\&\leq\int_{\O}\frac{|\m_{0}|^2}{\rho_{0}}\;dx+\int_{\O}\int_{\R^3}(1+|\xi|^{2})f_{0}\; d\xi dx
\end{split}
\end{equation}
\text{holds for almost everywhere} $t\in[0,T].$
 \end{Definition}

Our main result on existence of global weak solutions reads as follows.
 \begin{Theorem}\label{T1}
 Under the assumption \eqref{2.4}, for any $\gamma >\frac{3}{2}$, there exists a global weak solution $(\rho,\u,f)$ to the initial value problem \eqref{NSVB1}-\eqref{intial data} for any $T>0$.
 \end{Theorem}

\

We start  now to gather estimates for the momentum equation. Multiplying \eqref{NSVB2} by $\u$, integrating over $\O$, and using \eqref{NSVB1}, we deduce that
\begin{equation}\label{energy for NS}
\begin{split}
&\frac{d}{dt}\int_{\O}\frac{1}{2}\left(\rho|\u|^{2}+\frac{\rho^{\gamma}}{\gamma-1}\right)\;dx+\mu\int_{\O}|\nabla\u|^{2}\;dx
+\lambda\int_{\O}|\Dv\u|^2\,dx
\\&=-\int_a^b\int_{\O}\int_{\R^3}r f(\u-\xi)\cdot\u \,d\xi\; dx\,dr.
\end{split}
\end{equation}
Meanwhile, multiplying the Vlasov-Boltzmann equation \eqref{NSVB3} by $r^3\frac{|\xi|^{2}}{2}$, taking integration with respects to $r,\xi,x$, and using integration by parts, one obtains

\begin{equation}\begin{split}\label{energy for VB}
&\frac{d}{dt}\int_a^b\int_{\O}\int_{\R^3}\frac{1}{2} r^3|\xi|^2 f\,d\xi\,dx\,dr-\int_a^b\int_{\O}\int_{\R^3}r(\u-\xi)\xi f\,d\xi\,dx\,dr
\\&\quad\quad\quad\quad\quad\quad\quad\quad\quad\quad\quad\quad\quad\quad=\int_a^b\int_{\O}\int_{\R^3}r^3|\xi|^2 Q(f)\,d\xi\,dx\,dr.
\end{split}\end{equation}
Thus, from \eqref{energy for NS} and \eqref{energy for VB},
the following energy equality holds
\begin{equation}
\label{energy inequality-1}
\begin{split}
&\frac{d}{dt}\int_{\O}\left(\rho|\u|^{2}+\frac{\rho^{\gamma}}{\gamma-1}\right)\;dx+\frac{d}{dt}\int_a^b\int_{\O}\int_{\R^3} r^3|\xi|^2 f\,d\xi\,dx\,dr
\\&+2\mu\int_{\O}|\nabla\u|^{2}\;dx
+2\lambda\int_{\O}|\Dv\u|^2\,dx
+2\int_a^b\int_{\O}\int_{\R^3}r f(\u-\xi)^2 \,d\xi\; dx\,dr=0,
\end{split}
\end{equation}
where we used the following equality
$$\int_a^b\int_{\O}\int_{\R^3}r^3|\xi|^2Q(f)\,d\xi\,dx\,dr=0.$$
In fact, the last identity is obtained from the following Lemma \ref{Lemma of operator}(setting $p=2$), that uses the properties \uppercase\expandafter{\romannumeral2}-\uppercase\expandafter{\romannumeral5} on $Q(f)$ from hypotheses A.
\begin{Lemma}
\label{Lemma of operator} Under the properties \uppercase\expandafter{\romannumeral2}-\uppercase\expandafter{\romannumeral5} on $Q(f)$ from hypotheses A, then
for any $p\geq 1$, we have
\begin{equation}
\label{identity for any p}
\int_a^b\int_{\O}\int_{\R^3}r^3|\xi|^p Q(f) \,d\xi\; dx\,dr=0.
\end{equation}
\end{Lemma}
\begin{proof}
\begin{equation*}
\begin{split}
&\int_a^b\int_{\O}\int_{\R^3}r^3|\xi|^pQ(f)\,d\xi\,dx\,dr=-\nu\int_a^b\int_{\O}\int_{\R^3}r^3|\xi|^p f(x,\xi,r,t)\,d\xi\,dx\,dr
\\&+\nu\int_a^b\int_{\O}\int_{\R^3}\int_{r^*>r}r^3|\xi|^pB(r^*,r) f(x,\xi,r^*,t)\,dr^*\,d\xi\,dx\,dr
\\&=-\nu\int_a^b\int_{\O}\int_{\R^3}r^3|\xi|^p f(x,\xi,r,t)\,d\xi\,dx\,dr
\\&+\nu\int_a^b\int_{\O}\int_{\R^3}|\xi|^p\left(\int_{r^*>r}r^3B(r^*,r)\,dr\right)f(x,\xi,r^*,t)\,dr^*\,d\xi\,dx.
\end{split}
\end{equation*}
From Leger-Vasseur\cite{LV}, one can see that the properties \uppercase\expandafter{\romannumeral2}-\uppercase\expandafter{\romannumeral5} on $Q(f)$ yield
$$\int_{r^*>r}r^3B(r^*,r)\,dr=(r^*)^3,$$
so replacing in the second term one obtains a  symmetrization property yielding the zero integral, hence yield \eqref{identity for any p} holds.
\end{proof}
\bigskip
Next we estimate the transport Vlasov-Boltzmann equation \eqref{NSVB3} multiplying  by $r^3$ and integrating with respects to $r,\xi,x$, and using integration by parts, one obtains that
\begin{equation}
\label{conservation}
\frac{d}{dt}\int_a^b\int_{\O}\int_{\R^3}r^3f (x,\xi,r,t)\,d\xi\,dx\,dr=0.\end{equation}
In fact, this was proved in \cite{LV}. Using \eqref{energy inequality-1} and \eqref{conservation}, one obtains the following energy identity
\begin{equation}
\label{energy inequality}
\begin{split}
&\frac{d}{dt}\int_{\O}\left(\rho|\u|^{2}+\frac{\rho^{\gamma}}{\gamma-1}\right)\;dx+\frac{d}{dt}\int_a^b\int_{\O}\int_{\R^3} r^3(|\xi|^2+1) f\,d\xi\,dx\,dr
\\&+2\mu\int_{\O}|\nabla\u|^{2}\;dx
+2\lambda\int_{\O}|\Dv\u|^2\,dx
+2\int_a^b\int_{\O}\int_{\R^3}r f(\u-\xi)^2 \,d\xi\; dx\,dr=0.
\end{split}
\end{equation}

\vskip0.3cm
\bigskip

\section{Regularization}
In  order to prove Theorem~\ref{T1}, motivated by the techniques developed by Feireisl-Novotn\'{y}-Petzeltov\'{a} \cite{FNP} and  the work of Feireisl \cite{F04}, we first regularize the system \eqref{operator}-\eqref{NSVB3} by perturbing both the mass and momentum equations,  \eqref{NSVB1} and  \eqref{NSVB2} respectively, by adding $\varepsilon$-viscous terms and the $\delta$-modified pressure as follows (for simplicity we will not denote the solutions $(\rho, \u, f)$ dependance on the parameters $\eps$ and $\delta$ in this section)
\begin{equation}
\begin{split}
\label{approximation-1}
&\rho_t+\Dv(\rho\u)=\varepsilon \D\rho,
\\&
(\rho\u)_{t}+\Dv(\rho\u\otimes \u)+\nabla \rho^{\gamma}+\delta\nabla\rho^{\beta}-\mu\D\u-\lambda\nabla\Dv\u-\varepsilon\nabla\u\cdot\nabla\rho+\mathfrak{n}\u=j,
\\&
f_{t}+\xi\cdot\nabla_{x}f+\Dv_{\xi}{\left(\frac{(\u-\xi)f}{r^2}\right)}=Q(f),
\end{split}
\end{equation}
where $$\mathfrak{n}(t,x)=\int_a^b\int_{\R^3}r f\,d\xi\,dr,\;\;j=\int_a^b\int_{\R^3}r \xi f\,d\xi\,dr,$$ and $Q(f)$ is given by \eqref{operator}.

Assume initial data $(\rho_0,\u_0,f_0)$ satisfying
\begin{equation}
\label{regularized initial data}
\begin{split}
&\rho(0)=\rho_0(x)\in C^{2+\nu}(\bar{\O}), \quad\quad0<\underline{\rho}\leq \rho_0\leq \bar{\rho},
\\&(\rho\u)(0)=\m_0,\quad\quad \m_0=(m_0^1,m_0^2,m_0^3),\quad \text{ where } m_0^i\in C^2(\bar{\O}),
\\& f(0)=f_0(x,\xi,r),\quad f_0\geq 0, \;\; f_0\in L^{\infty}(\O\times\R^3\times R^+)\cap L^1(\O\times\R^3\times R^+)\;
\\&\text{and it is compactly supported with respects to } r, \xi.
\end{split}
\end{equation}

In order to solve $\varepsilon, \delta$-regularized Navier-Stokes part  of system \eqref{operator}-\eqref{NSVB3},  we need to show  that first moment $j(x,t)$ is bounded in $L^p(0,T;L^q(\O))$, for some $p,q>1$,  where of the  $j(x,t)$, the solution for  Vlasov-Boltzmann transport equation kinetic equation  \eqref{NSVB3}, is a source term in the $\varepsilon,\delta$-regularized momentum equation of Navier-Stokes part  of system.

 In addition,  the compressible $\varepsilon,\delta$-regularized Navier-Stokes part can be solved by using the  approximate by finete dimensional spaces arguments as in  Feireisl-Novotn\'{y}-Petzeltov\'{a} \cite{FNP} and  Feireisl \cite{F04} for fluid systems models as follows.

 We define the following  finite dimensional Banach space $X_k=\text{span}\{e_1,e_2,....,e_k\}$, for  $n\in \mathbb{N}$, and  each $e_i$ is an orthogonal basis of $L^2(\O)$, which is also an orthogonal basis of $H^2(\O).$

 In particular, $e_i$ could be chosen by  $-\D e_i=\lambda_i e_i.$, that is eigenfuctions of the Laplace operator acting  over the domain $\O$.

 Thus, without loss of generality, we consider an  infinite sequence of finite dimensional spaces
 \begin{equation}\label{Xk}
 X_k=\text{span}\{e_i\}_{i=1}^k,\quad k=1,2,3... , \, ,
 \end{equation}
  and will construct a sequences of triples $(\rho_k,\u_k,f_k)$ solutions of the following $k,\varepsilon$-approximate problem:

{\bf Step 1:}\  Starting from $\u_{k-1} $  given in $ C([0,T];X_{k-1}),$  where $X_{k-1}=\text{span}\{e_1,e_2,....,e_{k-1}\}$  solve the following initial value problem for the Vlasov-Boltzmann transport equation \eqref{NSVB3}.

For any  $f_0\in L^{\infty}(\O\times\R^3\times\R^+)\cap L^1(\O\times\R^3\times\R^+)$ with
 $f_0\geq 0,$ and $\supp f_0\subset \O\times\R^3$, solve the
 Vlasov-Boltzmann transport equation
 \begin{equation}\begin{split}
\label{kinetic}
&\partial t f_k+\xi\cdot\nabla_x f_k+\Dv_{\xi}\left(\frac{\u_{k-1}-\xi}{r^2}f_k\right)=Q(f_k)(x,\xi,r,t), \ \ \ \forall\ t>0\, , \\
&\quad f_k(x,\xi,r,0)=f_0(x,\xi,r) \ \ \text{for all}     \ \   (x,\xi,r) \ \in   \O\times\R^3\times\R^+\ .
\end{split}\end{equation}
and show the the first moment  $j_k(x,t)= \int (\xi,r) f_k(x,\xi,r,t) d \xi dr$  associated to
is bounded in  $ L^{\infty}(0,T;L^2(\O)).$

\

{\bf Step 2:} For any initial data density-velocity pair $(\rho_k,\u_k)(x,0)$ satisfying
$\rho_k \in L^\gamma(X_k))$, $\u_k \in  L^2(X_k) $ and $\nabla\u_k \in L^2(X_k), $
there is a unique weak $k,\varepsilon$- approximate solution triple $\rho_k \in L^\infty([0,T]; L^\gamma(X_k) )$, $\u_k \in L^\infty([0,T]; L^2(X_k)  )$ and $\nabla\u_k \in  L^2([0,T]; L^2(X_k)  )$
satisfying  the integral equation
\begin{equation}
\begin{split}
\label{integral equation-1}
&\int_{\O}\rho\u_k(t)\cdot\varphi\,dx-\int_{\O}\m_0\cdot\varphi\,dx
=\int_0^t\int_{\O}\left(\mu\D\u_k+\lambda\nabla\Dv\u_k\right)\varphi\,dx\,dt
\\&+\int_0^T\int_{\O}\left(\varepsilon\nabla\u_k\cdot\nabla\rho-\Dv(\rho\u_k\otimes \u_k)-\nabla \rho^{\gamma}-\delta\nabla\rho^{\beta}-\mathfrak{n}\u_k+j\right)\varphi\,dx\,dt
\end{split}
\end{equation}
for any test function $\varphi\in X_k$.

\begin{Proposition}\label{k-approx-problem}
For any initial data $(\rho_0,\u_0)(x,0)$ with $\rho_0 \in L^\gamma(\O))$, $\u_0 \in  L^2(\O) $ and $\nabla\u_0 \in L^2(\O), $,  and  $f_0\in L^{\infty}(\O\times\R^3\times\R^+)\cap L^1(\O\times\R^3\times\R^+)$, there exits a weak solution to system \eqref{integral equation-1}-\eqref{kinetic} denoted by the triple $(\rho_k,\u_k, f_k)$ in the spaces
$L^\infty([0,T]; L^\gamma(X_k) \times  L^\infty([0,T]; L^2(X_k)  ) \times (f_0\in L^{\infty}(\O\times\R^3\times\R^+)\cap L^1(\O\times\R^3\times\R^+))$.

In addition the triple components are {\em uniformly bounded} in the $k$ and $\varepsilon$ parameters.
\end{Proposition}

The proof of Proposition~\ref{k-approx-problem} is rather elaborated and will be done in several steps that gather the necessary estimates. So we start proving or recalling the following results

\bigskip

The first result towards addressing the {{\bf Part 1} of the $k$-iteration argument, was proved in Leger-Vasseur \cite{LV}.
 \begin{Proposition}
 \label{proposition of kinetic equation}
 For any given $\u \in C([0,T],C(\O))$, there exist a unique non-negative weak solution to the kinetic problem \eqref{kinetic} for any $T>0$ , provided the initial data satisfies
 $$f_0\in L^{\infty}(\O\times\R^3\times\R^+)\cap L^1(\O\times\R^3\times\R^+)$$ and
 $$f_0\geq 0,\;\;\supp f_0\subset \O\times\R^3, $$
 that is,  $f(x,\xi,r,t)$ satisfies
 \begin{equation}
\label{weak kinetic equation}
\begin{split}
&\int_0^T\int_{\R^{+}\times\R^6}f\left(\varphi_t+\xi\cdot\nabla_x\varphi-\frac{\u-\xi}{r^2}\cdot\nabla_{\xi}\varphi\right)\,dx\,d\xi\,d r+
\int_0^T\int_{\R^{+}\times\R^6}Q(f)\varphi\,dx\,d\xi\,dr
\\&+
\int_0^T\int_{\R^{+}\times\R^6}f^0\varphi(0,x,\xi,r)\,dx\,d\xi\,dr=0
\end{split}
\end{equation}
for any test function $\varphi(t,x,\xi,r).$

 Moreover, this non-negative weak solution satisfies
 the following estimates:
 \begin{equation}
 \label{estimate of f}
 \begin{split}
 &f\in L^{\infty}(0,T;L^1(\O\times\R^3\times\R^+)),
  \\&f\in L^{\infty}(0,T;L^{\infty}(\O\times\R^3\times\R^+)),
  \end{split}
  \end{equation}
 $$f\in C([0,T];W^{-1,p}(\O\times\R^3\times\R^+)),\quad\text{ for any } 1\leq p \leq \infty,$$
 $$supp(f)\subset \O\times\R^3\text{ for a.e. } t\in [0,T].$$
 \end{Proposition}

 \bigskip

The next step is to secure that the weak solution $f_k(x,\xi,r,t)$ constructed in Proposition~\ref{proposition of kinetic equation} has if kinetic first moment $j_k(x,t) \in  L^{\infty}(0,T;L^2(\O)).$

Such estimates are a result of the following proposition, whose proof follows immediately.
\begin{Proposition}\label{pro of f}
If $ \u_k\in C([0,T];X_{k})$, then there exist operators $n_k=N(\u_k),j= L(\u_k):C([0,T];X_{k})\rightarrow C([0,T];C(\O))$ satisfying

\

\begin{itemize}

  \item[{\bf i)}] (Lipschitz estimate for the kinetic density) \begin{equation*}
\|n_k^1-n_k^2\|_{L^{\infty}(0,T;L^{\infty}(\O)}\leq C(a,b,T)\|\u_k^1-\u_k^2\|_{L^{2}(0,T;L^2(\O))},
\end{equation*}.

\

 \item[{\bf ii)}] (Lipschitz estimate for the mean velocity) \begin{equation*}
\|j_k^1-j_k^2\|_{L^{\infty}(0,T;L^{\infty}(\O)}\leq C(a,b,T)\|\u_k^1-\u_k^2\|_{L^{2}(0,T;L^2(\O))},
\end{equation*} \\ for any $\u_k^{1},\u_k^{2}$ in the following set $$M_{L}=\{\u_k\in C([0,T];X_{k});\|u\|_{C([0,T];X_{k})}\leq L,\ t\in\ [0,T]\}.$$
\end{itemize}
\end{Proposition}

\begin{proof}
Following the strategy of Leger-Vasseur \cite{LV}, we construct a sequence of solutions verifying

\begin{equation}
\begin{cases}
\label{kinetic approximation}
&\partial_t f_k+\xi\cdot\nabla f_k+\Dv_{\xi}\left(\frac{\u_{k-1}-\xi}{r^2}f_n\right)=-\nu f_k(x,\xi,r,t)
\\&\quad\quad\quad\quad\quad\quad\quad\quad+\nu\int_{r>r^*}B(r^*,r)f_{k-1}(x,\xi,r^*,t)\,d r^*,\\
& f_k(x,\xi,r,0)=f_0(x,\xi,r).
\end{cases}
\end{equation}
as follows. First, we need to write the following ODEs:
\begin{equation}
\begin{cases}
\label{ode}
&\frac{dx}{dt}=\xi;\\
&\frac{d\xi}{dt}=\frac{\u_{k-1}-\xi}{r^2};\\
&x(0)=x;
\\&\xi(0)=\xi,
\end{cases}
\end{equation}
  then, by the characteristic method, we have the following solution to \eqref{kinetic approximation}
\begin{equation}
\label{kinetic solution}
\begin{split}
&f_k(t,x,\xi,r)=e^{-\int_0^t(\nu -\frac{3}{r^2})d\,s}f_0(x(0,t,x,\xi),\xi(0,t,\tau),r)
\\&+\nu\int_0^t\int_{\R^+}e^{-\int_0^t(\nu -\frac{3}{r^2})d\,s}B(r,r^*)f_{k-1}(\tau,x(\tau,t,x,\xi),r^*)d r^*\,d\tau.
\end{split}
\end{equation}
So taking the limits as $k\to\infty$, one obtains the weak solutions to \eqref{kinetic} by the standard argument of weak convergence in Leger-Vasseur \cite{LV}. However, we need to use \eqref{kinetic solution} to
derive some new estimates due to the compressible fluids and the coupling to the kinetic equations.  Let $f_k^1$ and $f_k^2$ be two solutions to  \eqref{kinetic approximation} corresponding to $\u_{k-1}^1$ and $\u_{k-1}^2$ respectively, and
$f^1$ and $f^2$ be two weak solutions to  \eqref{kinetic} corresponding to $\u^1$ and $\u^2$ respectively.
 Letting $Y(t,x,\xi)=(x,\xi)$, we have
\begin{equation}
\label{difference of f n}
\begin{split}
&\|f_k^1-f_k^2\|_{L^{\infty}(0,T;L^{\infty}(\O\times\R^3\times\R^+))}
\\&\leq C(T)\|Y_1-Y_2\|_{L^{\infty}(0,T;L^{\infty}(\O\times\R^3\times\R^+))}+C(T)\int_0^t\|f_k^1-f_k^2\|_{L^{\infty}(0,T;L^{\infty}(\O\times\R^3\times\R^+))}\,ds.
\end{split}
\end{equation}
As Leger-Vasseur \cite{LV},
  \begin{equation}
  \label{convergence of density function}
  f_k\to f\text{   in }L^{p}(0,T;L^{p}(\O\times\R^3\times\R^+))\;\;\text{ and } f\in L^{\infty}(0,T;L^{\infty}(\O\times\R^3\times\R^+)).
  \end{equation}
  Letting $k\to\infty$ in \eqref{difference of f n}, yields
\begin{equation}
\label{difference of f}
\begin{split}
&\|f^1-f^2\|_{L^{\infty}(0,T;L^{\infty}(\O\times\R^3\times\R^+))}
\\&\leq C(T)\|Y_1-Y_2\|_{L^{\infty}(0,T;L^{\infty}(\O\times\R^3\times\R^+))}+C(T)\int_0^t\|f^1-f^2\|_{L^{\infty}(0,T;L^{\infty}(\O\times\R^3\times\R^+))}\,ds.
\end{split}
\end{equation}
However, in our case we need to control the characteristic ODE's of the transport flow depending on $\u_k(x,t)$, that we estimate as follows.

The first term above, after using \eqref{ode} with $u_{k-1}$, can be estimated by
\begin{equation*}
\begin{split}
&\|Y_1-Y_2\|_{L^{\infty}(0,T;L^{\infty}(\O\times\R^3\times\R^+))}
\\&\leq C \left(\int_0^t\|\frac{\u_{k-1}^1-\u_{k-1}^2}{r^2}\|_{L^{\infty}(\O)}\,ds+\int_0^t(1+\|\frac{\u_{k-1}^1}{r^2}\|_{W^{1,\infty}(\O)})\|Y_1-Y_2\|_{L^{\infty}(\O\times\R^3\times\R^+)}\,ds\right),
\end{split}\end{equation*}
and so by Gronwall inequality, we obtain
\begin{equation}
\label{difference of Y1 Y2}
\|Y_1-Y_2\|_{L^{\infty}(0,T;L^{\infty}(\O\times\R^3\times\R^+))}\leq C(a,b,T)\int_0^t\|\u_{k-1}^1-\u_{k-1}^2\|_{L^2(\O)}\,ds.
\end{equation}
In addition, by \eqref{difference of f} and \eqref{difference of Y1 Y2},
\begin{equation}
\label{key estimate of difference f}
\|f_k^1-f_k^2\|_{L^{\infty}(0,T;L^{\infty}(\O\times\R^3\times\R^+))}\leq C(a,b,T)\|\u_{k-1}^1-\u_{k-1}^2\|_{L^2(0,T;\O)}.
\end{equation}
Let $\mathfrak{n}_k=N(\u_{k-1})$ and $j_k=L(\u_{k-1})$,  from \eqref{key estimate of difference f}, one obtains the following two estimates
\begin{equation}
\label{continuity of N(u)}
\|\mathfrak{n}_k^1 -\mathfrak{n}_k^2\|_{L^{\infty}(0,T;L^{\infty}(\O))}=
\|N(\u_{k-k}^1)-N(\u_{k-1}^2)\|_{L^{\infty}(0,T;L^{\infty}(\O))}\leq C(a,b,T)\|\u_{k-1}^1-\u_{k-1}^2\|_{L^{2}(0,T;L^2(\O))},
\end{equation}
and
\begin{equation}
\label{continuity of L(u)}
\|j_k^1-j_k^2\|=\|L(\u_{k-1}^1)-L(\u_{k-1}^2)\|_{L^{\infty}(0,T;L^{\infty}(\O))}\leq C(a,b,T)\|\u_{k-1}^1-\u_{k-1}^2\|_{L^{2}(0,T;L^2(\O))}.
\end{equation}
The proof of Proposition \ref{pro of f} is completed, and we have all needed estimates to complete {\bf Part 1.} of the iteration needed to construct the solutions stated in Proposition~\ref{k-approx-problem}

\end{proof}

For {\bf Part 2.} of the iteration, it is natural to obtain an energy identity for the kinetic part  \eqref{NSVB3} of  $k,\varepsilon$-approximate compressible fluid kinetic system. The following proposition yields such identity.

\begin{Proposition}\label{energy equality for kinetic part} [Kinetic energy conservation]
If $ \u\in C([0,T];X_{k})$,  any weak solution $f$ of \eqref{kinetic} satisfies the following identity:
\begin{equation*}
\begin{split}&
\int_{\O}\int_a^b\int_{\R^3}r^3(1+|\xi|^2)f\,d\xi\,dr\,dx-\int_{\O}\int_a^b\int_{\R^3}r^3(1+|\xi|^2)f_0\,d\xi\,dr\,dx
\\&\quad\quad\quad\quad\quad\quad\quad\quad=2\int_0^t\int_{\O}\int_a^b\int_{\R^3}r(\u_{k-1}-\xi)f\xi\,d\xi\,dr\,dx\,dt.
\end{split}
\end{equation*}
\end{Proposition}
\begin{proof}
Using $1+|\xi|^2$ to multiply on both sides of \eqref{kinetic approximation}, and taking integration by parts, we have
\begin{equation}
\begin{split}
\label{AAAA}
&\int_{\O}\int_a^b\int_{\R^3}r^3(1+|\xi|^2)f_k\,d\xi\,dr\,dx-\int_{\O}\int_a^b\int_{\R^3}r^3(1+|\xi|^2)f^0_k\,d\xi\,dr\,dx
\\&=2\int_0^t\int_{\O}\int_a^b\int_{\R^3}r(\u_{k-1}-\xi)f_k\xi\,d\xi\,dr\,dx\,dt
\\&-\nu\int_0^t\int_{\O}\int_a^b\int_{\R^3}r^3(1+|\xi|^2)f_k\xi\,d\xi\,dr\,dx\,dt
\\&+\nu\int_0^t\int_{\O}\int_a^b\int_{\R^3}\int_{r>r^*}r^3(1+|\xi|^2)f_{k-1}(x,\xi,r^*,t)\,dr^*\xi\,d\xi\,dr\,dx\,dt.
\end{split}
\end{equation}
Letting $k\to\infty$ in \eqref{AAAA},
thanks to \eqref{convergence of density function} the Fubini's theorem, the conclusion can be followed.
\end{proof}
\bigskip

\begin{Lemma}
\label{Lemma of moment}Let $\u\in L^r(0,T;L^{N+p}(\O))$ be fixed with any $1\leq r\leq \infty$ and $p\geq 1$. Assume that $f_0\in L^{\infty}(\O\times\R^3\times\R^+)\cap L^1(\O\times\R^3\times\R^+),$ $r^3|\xi|^pf_0\in L^1(\O\times\R^3\times\R^+)$, then the solution $f(x,\xi,r,t)$ of \eqref{kinetic} has the following estimate
\begin{equation}
\begin{split}
\label{crucial estimate}
&\int_a^b\int_{\O}\int_{\R^3}r^3|\xi|^p f \,d\xi\; dx\,dr
\\&\leq pC_{T,N,b}\left(\left(\int_a^b\int_{\O}\int_{\R^3}r^3|\xi|^p f_0 \,d\xi\,dx\,dr\right)^{\frac{1}{N+p}}+(\|f_0\|_{L^{\infty}(\O\times\R^3\times\R^+)}+1)\|\u_{k-1}\|_{L^r(0,T;^{N+p}(\O))}\right)^{N+p},
\end{split}
\end{equation}
for any $0\leq t\leq T.$
\end{Lemma}
\begin{proof}
For any $p\geq 1$, multiplying $r^3|\xi|^p$ on both sides of kinetic equation \eqref{kinetic approximation}, we have
\begin{equation}
\begin{split}
\label{BBBBA}
&\int_a^b\int_{\O}\int_{\R^3}r^3|\xi|^p f_k \,d\xi\; dx\,dr-\int_a^b\int_{\O}\int_{\R^3}r^3|\xi|^p f^0_k \,d\xi\; dx\,dr
\\&=p\int_0^t\int_a^b\int_{\O}\int_{\R^3}r(\u_{k-1}-\xi)f_k|\xi|^{p-1}\cdot\frac{\xi}{|\xi|} \,d\xi\; dx\,dr\,dt
\\&-\nu\int_0^t\int_{\O}\int_a^b\int_{\R^3}r^3|\xi|^p f_k\xi\,d\xi\,dr\,dx\,dt
\\&+\nu\int_0^t\int_{\O}\int_a^b\int_{\R^3}\int_{r>r^*}r^3|\xi|^pf_{k-1}(x,\xi,r^*,t)\,dr^*\xi\,d\xi\,dr\,dx\,dt.
\end{split}
\end{equation}
Letting $k\to\infty$ in \eqref{BBBBA},
\eqref{convergence of density function} and the Fubini's theorem gives us
\begin{equation}
\label{integral-differential inequality}
\begin{split}
&\int_a^b\int_{\O}\int_{\R^3}r^3|\xi|^p f_k \,d\xi\; dx\,dr-\int_a^b\int_{\O}\int_{\R^3}r^3|\xi|^p f^0_k \,d\xi\; dx\,dr
\\&
+p\int_0^t\int_a^b\int_{\O}\int_{\R^3}r|\xi|^pf \,d\xi\; dx\,dr\,dt=\int_0^TI(t)\,dt,
\end{split}
\end{equation}
where $$I(t)=p\int_a^b\int_{\O}\int_{\R^3}r|\xi|^{p-1}f\u_{k-1} \cdot\frac{\xi}{|\xi|} \,d\xi\; dx\,dr.$$

Thanks to H\"{o}lder inequality, we can control $I(t)$ as follows
\begin{equation}
\label{estimate on I(t)}
I(t)\leq p\|\u_{k-1}\|_{L^s(\O)}\left(\int_{\O}\left(\int_a^b\int_{\R^3}r|\xi|^{p-1}f\,d\xi\,dr\right)^{s'}dx\right)^{\frac{1}{s'}},
\end{equation}
where $$\frac{1}{s}+\frac{1}{s'}\leq 1.$$
For any $R>0$, we have
\begin{equation}
\begin{split}
\label{estimate on I(t)-2}
\int_a^b\int_{\R^3}&r|\xi|^{p-1}f\,d\xi\,dr=
\int_a^b\int_{|\xi|\leq R}r|\xi|^{p-1}f\,d\xi\,dr+\int_a^b\int_{|\xi|\geq R}r|\xi|^{p-1}f\,d\xi\,dr
\\&\leq b^2\|f\|_{L^{\infty}(\O\times\R^3\times\R^+)}\frac{R^{N+p-1}}{N+p-1}+\frac{1}{a^2R}\int_a^b\int_{|\xi|\geq R}r^3|\xi|^p f\,d\xi\,dr.
\end{split}
\end{equation}
Taking $s=N+p$ in \eqref{estimate on I(t)}, and
 $$R=\left(\int_a^b\int_{\R^3} r^3|\xi|^p f\,d\xi\,dr\right)^{\frac{1}{N+p}}>0$$
in \eqref{estimate on I(t)-2},
one obtains that
\begin{equation}
\label{bound of I(t)}
I(t)\leq p\|\u_{k-1}\|_{L^{N+p}}(\frac{1}{a^2}+\frac{b^2}{N+p-1}\|f\|_{L^{\infty}(\O\times\R^3\times\R^+)})\left(\int_a^b\int_{\R^3} r^3|\xi|^p f\,d\xi\,dr\right)^{\frac{N+p-1}{N+p}}.
\end{equation}

Thanks to \eqref{integral-differential inequality} and \eqref{bound of I(t)}, we deduce
\begin{equation*}
\begin{split}
&\int_a^b\int_{\O}\int_{\R^3}r^3|\xi|^p f \,d\xi\; dx\,dr
\\&\leq pC_{T,N,a,b}\left(\left(\int_a^b\int_{\O}\int_{\R^3}r^3|\xi|^3 f_0 \,d\xi\,dx\,\,dr\right)^{\frac{1}{N+p}}+(\|f_0\|_{L^{\infty}(\O\times\R^3\times\R^+)}+1)\|\u_{k-1}\|_{L^r(0,T;L^{N+p}(\O))}\right)^{N+p}
\end{split}
\end{equation*}
for any $0\leq t\leq T.$
\end{proof}

\bigskip
We are now in conditions  to obtain estimates for the zero moment \eqref{zero moment} and first moment \eqref{first moment} of the solutions of the Vlasov-Boltzman equations \eqref{kinetic}.
 We estimate these quantities  in the following lemma \ref{Lemma of n-j}  that may be similar to the variation of the classical regularity of moments, see \cite{LP}. The proof closely follows the argument as in \cite{H}.
\begin{Lemma}
\label{Lemma of n-j} Under the hypothesis of Lemma \ref{Lemma of moment},
for any $p\geq 1$, $0\leq t\leq T$, we have

\begin{equation}\label{moment0}
\|\mathfrak{n}_k\|_{L^{\frac{N+p}{N}}(\O)} \leq C_{N,b,T}(\|f_k\|_{L^{\infty}(\O\times\R^3\times\R^+)}+1)\left(\int_a^b\int_{\O}\int_{\R^3}r^3|\xi|^pf_k\,d\xi\,dx\,dr\right)^{\frac{N}{N+p}},
\end{equation}
and
\begin{equation}\label{moment1} \|j_k\|_{L^{\frac{N+p}{N+1}}(\O)} \leq C_{N,b,T}(\|f_k\|_{L^{\infty}(\O\times\R^3\times\R^+)}+1)\left(\int_a^b\int_{\O}\int_{\R^3}r^3|\xi|^pf_k\,d\xi\,dx\,dr\right)^{\frac{N+1}{N+p}}.
\end{equation}
\end{Lemma}

\begin{proof} For any $R>0$,
we can estimate $n$ as follows
\begin{equation}
\label{split n}
\begin{split}
\mathfrak{n}(t,x)&=\int_a^b\int_{\R^3}r f\,d\xi\,dr=\int_a^b\int_{|\xi|\leq R}r f\,d\xi\,dr+\int_a^b\int_{|\xi|\geq R}r f\,d\xi\,dr
\\&\leq bR^N\|f\|_{L^{\infty}(\O\times\R^3\times\R^+)}+\frac{1}{a^2 R^p}\int_a^b\int_{|\xi|\geq R}r^3|\xi|^pf\,d\xi\,dr.
\end{split}
\end{equation}
Taking
 $$R=\left(\int_a^b\int_{\R^3}r^3|\xi|^pf\,d\xi\,dr\right)^{\frac{1}{N+p}}$$
 which is finite by Lemma \ref{Lemma of moment}, depending only on the initial data, yields
\begin{equation*}
\mathfrak{n}(t,x)\leq C_{N,b}(\|f\|_{L^{\infty}(\O\times\R^3\times\R^+)}+\frac{1}{a^2})\left(\int_a^b\int_{\R^3}r^3|\xi|^pf\,d\xi\,dr\right)^{\frac{N}{N+p}},
\end{equation*}
and since the estimate \ref{crucial estimate} is uniform in $[0,T]$, thus
\begin{equation*}
\|\mathfrak{n}(t,x)\|_{L^{\infty}(0,T;L^{\frac{N+p}{N}}(\O))}\leq C_{N,b,T}(\|f\|_{L^{\infty}(\O\times\R^3\times\R^+)}+\frac{1}{a^2})\left(\int_{\O}\int_a^b\int_{\R^3}r^3|\xi|^pf\,d\xi\,dr\,dx\right)^{\frac{N}{N+p}}.
\end{equation*}
We can also use the same arguments to show
$$\|j\|_{L^{\infty}(0,T;L^{\frac{N+p}{N+1}}(\O))} \leq C_{N,b,T}(\|f\|_{L^{\infty}(\O\times\R^3\times\R^+)}+1)\left(\int_{\O}\int_a^b\int_{\R^3}r^3|\xi|^pf\,d\xi\,dr\,dx\right)^{\frac{N+1}{N+p}}.$$

\end{proof}

\vskip0.3cm
\bigskip

\bigskip

 Since eigenfunctions of $$\Delta e_i=\lambda_{i}e_i\quad\text{ in } \O$$have bounded solutions,
  then
  $$\u\in L^{2}(0,T;L^{\infty}(\O)).$$
 In particular, such estimate allows us to apply Lemma \ref{Lemma of moment} to obtain
  \begin{equation}
 \label{p=5}
\int_a^b\int_{\O}\int_{\R^3}r^3|\xi|^5 f\,d\xi\,dx\,dr<\infty,
\end{equation}
  provided the initial data satisfies
\begin{equation*}
\int_a^b\int_{\O}\int_{\R^3}r^3|\xi|^p f_0\,d\xi\,dx\,dr<\infty
\end{equation*}
for any $p\geq 5.$ Therefore,
Applying Lemma \ref{Lemma of n-j}  to get estimate to the corresponding first moment of the solution of the kinetic equation to \eqref{p=5} with $p=5$ and $N=3$, we obtain
\begin{equation}
 \label{j L^2}
 \mathfrak{n}=N(\u)\in L^{\infty}(0,T; L^{\frac{8}{3}}(\O)),\quad j=L(\u)\in L^{\infty}(0,T;L^2(\O)),
 \end{equation}
 and satisfy the estimates \eqref{continuity of N(u)} and \eqref{continuity of L(u)}.
As a consequence we are able to solve the following regularized compressible Navier-Stokes part by using the estimate on the first kinetic moment $j(t,x)$ of the system
\begin{equation}
\begin{split}
\label{approximation-NS part}
&\rho_t+\Dv(\rho\u)=\varepsilon \D\rho,
\\&
(\rho\u)_{t}+\Dv(\rho\u\otimes \u)+\nabla \rho^{\gamma}+\delta\nabla\rho^{\beta}-\mu\D\u-\lambda\nabla\Dv\u-\varepsilon\nabla\u\cdot\nabla\rho+N(\u)\u=j,
\end{split}
\end{equation}
with the initial data \eqref{regularized initial data}.

\vskip0.3cm
In fact, we notice that $n\u$ is a good term for the compressible Navier-Stokes equations because $n(t,x)\geq 0$ is on the left side of the momentum equation and so it is active as an absorbing term that stabilized the momentum flow. Another advantage is that the right hand side $j(x,t)$ is bounded in  $ L^{\infty}(0,T;L^2(\O)).$
Thus the weak solution $(\rho,\u)$ to \eqref{approximation-NS part} can be constructed following the now classical approach in
 Feireisl-Novotn\'{y}-Petzeltov\'{a} \cite{FNP} and  Feireisl \cite{F04} for fluid equations.  In fact,
 we can find the approximate solutions $\u_k \in C([0; T];X_k)$
satisfy the integral equation
\begin{equation}
\begin{split}
\label{integral equation-1}
&\int_{\O}\rho\u_k(t)\cdot\varphi\,dx-\int_{\O}\m_0\cdot\varphi\,dx
=\int_0^t\int_{\O}\left(\mu\D\u_k+\lambda\nabla\Dv\u_k\right)\varphi\,dx\,dt
\\&+\int_0^T\int_{\O}\left(\varepsilon\nabla\u_k\cdot\nabla\rho-\Dv(\rho\u_k\otimes \u_k)-\nabla \rho^{\gamma}-\delta\nabla\rho^{\beta}-\mathfrak{n}\u_k+j\right)\varphi\,dx\,dt
\end{split}
\end{equation}
for any test function $\varphi\in X_k$.
\vskip0.3cm

In order to solve \eqref{integral equation-1}, we follow the same arguments as in  \cite{FNP, F04}, and introduce the following two operators that are crucial to apply fixed point arguments later by generating an ODE in a suitable Banach space.

 In our case, the iteration map for a fixed point argument is constructed as follows.
 For any given $\u\in C([0,T];X_k)$, $\rho$ is a solution to the following problem
 \begin{equation}\label{b2}
\begin{aligned}
\left\{\begin{array}{lll}
\partial_{t}\rho+\rm{div}(\rho u)=\varepsilon \triangle\rho,\\
\rho_{0}\in C^{\infty}(\mathbb{T}^{3}),\ \ \rho_{0}\geq \underline{\rho}>0.
\end{array}\right.
\end{aligned}
\end{equation}
First, we introduce the operator $\mathcal{S}$ as follows $$\mathcal{S}:C([0,T];X_{k})\rightarrow C([0,T];C(\O)), \rho=\mathcal{S}(\u),$$ and recall the following two Propositions that can be found in \cite{FNP}

\begin{Proposition}\label{pro1}
If $0<\underline{\rho}\leq \rho_0\leq \overline{\rho},\ \rho_{0}\in C^{\infty}(\O),\ \u\in C([0,T];X_{k})$, then there exists an operator $\mathcal{S}:C([0,T];X_{k})\rightarrow C([0,T];C(\O))$ satisfying

\

\begin{itemize}
  \item [{\bf i)}]$\rho=\mathcal{S}(\u)$ is an unique solution to the problem (\ref{b2}).

  \

  \item   [{\bf ii)}]  Density bounds:
\begin{equation}
  \label{boundness of density}
  0<\underline{\rho}e^{-\int_{0}^{T}\|{\rm div}\u\|_{L^{\infty}}dt}\leq \rho(x,t)\leq \overline{\rho}e^{\int_{0}^{T}\|{\rm div}\u\|_{L^{\infty}}dt},\ \ \text{for any } x\in \O,\ \ t\geq0.
  \end{equation}

\

  \item  [{\bf iii)}] { Lipchitz condition:} \begin{equation}
  \label{contunity of S(u)}
\|\mathcal{S}(u_{1})-\mathcal{S}(u_{2})\|_{C([0,T];C(\O))}\leq  T C( \rho_{0},\varepsilon,L)
\|\u_{1}-\u_{2}\|_{C([0,T];X_{k})},\end{equation} \\ for any $\u_{1},\u_{2}$ in the following set $$M_{L}=\{\u\in C([0,T];X_{k});\|\u\|_{C([0,T];X_{k})}\leq L,\ t\in\ [0,T]\}.$$
\end{itemize}
\end{Proposition}

\vskip0.3cm

 In addition, for any given function $\rho\in C^{1}(\O)$ with $\rho \geq\underline{\rho}>0$, we introduce an operator $\mathcal{M}$ for fixed $t$, satisfying
$$\mathcal{M}[\rho]:X_{k}\rightarrow X_{k}^{*},\ \ <\mathcal{M}[\rho] \u,\v>=\int_{\O}\rho \u\cdot \v\,dx,\ \ \text{ for any }\u,\ \v\in X_{k},$$
and we recall from \cite{FNP}, (page 363-364)  the following proposition describing the properties of $\mathcal{M}$:

\begin{Proposition}\label{pro2}
 For any given function $\rho\in C^0(0,T;C^{1}(\O))$ with $\rho \geq\underline{\rho}>0$, where $\underline{\rho}$ is a constant,
\begin{itemize}
  \item  [{\bf i)}] $\|\mathcal{M}[\rho]\|_{\mathcal{L}(X_{k},X_{k}^{*})}\leq C(k)\|\rho\|_{L^{1}}.$

 \

  \item[{\bf ii)}]$\|\mathcal{M}[\rho]\|_{\mathcal{L}(X_{k},X_{k}^{*})}\geq \inf_{x\in\O}\rho$

 \

    \item  [{\bf iii)}]  If $\inf_{x\in\O}\rho\geq \underline{\rho}>0$, then the operator is invertible with
      $$\|\mathcal{M}^{-1}[\rho]\|_{\mathcal{L}(X_{k}^{*},X_{k})}\leq \underline{\rho}^{-1},$$
      where $\mathcal{L}(X_{k}^{*},X_{k})$ is the set of bounded liner mappings from $X_{k}^{*}$ to $X_{k}$.

      \

   \item  [{\bf iv)}] $\mathcal{M}^{-1}[\rho]$ is Lipschitz continuous in $X_k^*$ in the sense
  \begin{equation}
  \label{contunity of inverse function}
  \|\mathcal{M}^{-1}[\rho_{1}]-\mathcal{M}^{-1}[\rho_{2}]\|_{\mathcal{L}(X_{k}^{*},X_{k})}\leq C(n,\underline{\rho})\|\rho_{1}-\rho_{2}\|_{L^{1}(\O)}
  \end{equation}
  for all $\rho_{1},\rho_{2}\in C^0(0,T;L^{1}(\O))$ such that $\rho_{1},\rho_{2}\geq \underline{\rho}>0$.
\end{itemize}
\end{Proposition}

The proofs of these two propositions can be found on   \cite{FNP}  (page 363)   and (page 363-364) respectively. They are sufficient in order to show the needed compactness for the existence of a fixed point solution set.


\vskip0.8cm

We apply the  strategy of  \cite{FNP}  to the problem under consideration, namely the existence of solutions to the coupled  compressible fluid equation to the gas kinetic equation, done through
 the gas density $\mathfrak{n}$ defined by \eqref{zero moment} and gas current $j$ defined by \eqref{first moment}.

 Indeed, making  use of the operators $\mathcal{M}[\rho]$, $\rho=\mathcal{S}(\u_{k})$, $\mathfrak{n}=N(\u_k)$ and $j=L(\u_k)$,
 we rewrite \eqref{integral equation-1} as the following ordinary differential equation on the finite-dimensional space $X_k$:
\begin{equation}\label{finite dimensional ode}
\begin{aligned}
&\frac{d}{dt}\left(\mathcal{M}[\mathcal{S}(\u_k)(t)]\u_k(t)\right)=\mathcal{N}
(\mathcal{S}(\u_k),N(\u_k),L(\u_k),\u_k), \quad t>0,
\\& \mathcal{M}[\mathcal{S}(\u_k)(0)]\u_k(0)=\mathcal{M}[\rho_0]\u_0,
\end{aligned}
\end{equation}
where
\begin{equation*}
\begin{split}
[\mathcal{N}
(\mathcal{S}(\u_k),N(\u_k),L(\u_k),\u_k),\varphi]&=\int_{\O}\left(\mu\D\u_k+\lambda\nabla\Dv\u_k+
\varepsilon\nabla\u_k\cdot\nabla\rho\right)\cdot\varphi\,dx\\&
-\int_{\O}\left(\Dv(\rho\u_k\otimes \u_k)+\nabla \rho^{\gamma}+\delta\nabla\rho^{\beta}+\mathfrak{n}\u_k-j\right)\cdot\varphi\,dx,
\end{split}
\end{equation*}
for all $\varphi\in X_k.$
 Integrating \eqref{finite dimensional ode} over $(0,t)$, we can write the problem as the following nonlinear problem:
\begin{equation}\label{fixed point}
\begin{aligned}
\u_k(t)=\mathcal{M}^{-1}[\mathcal{S}(\u_k)(t)]\big(\mathcal{M}[\rho_0]\u_0
+\int_{0}^{T}\mathcal{N}
(\mathcal{S}(\u_k),N(\u_k),L(\u_k),\u_k)(s)ds\big).
\end{aligned}
\end{equation}
Since $\mathcal{N}
(\mathcal{S}(\u_k),N(\u_k),L(\u_k),\u_k)$ is a Liptzchiz function, as all its argument from \eqref{continuity of N(u)}, \eqref{continuity of L(u)}, \eqref{contunity of S(u)} and \eqref{contunity of inverse function}, this equation can be solved with the fixed-point theorem of Banach, at least on
a small time $0<T'\leq T.$ Thus, we obtained $\u_k\in C^0(0,T';X_k).$

\bigskip

 In order to extend the existence final time in order to get $T'=T,$ it is enough to show there exists uniform estimates on solution triple $(\rho_k,\u_k,f_k)$ in suitable functional spaces defined over the finite dimensional space $X_k$.

 Indeed, the following definition of a suitable energy functional and subsequent proposition provide the global in time existence of solutions to the approximation  system \eqref{approximation-1}-\eqref{regularized initial data}.

 We first define the following energy functional associate to solutions of system \eqref{approximation-1}-\eqref{regularized initial data}.

  \begin{Definition}[{\sl The Energy Functional}]
 \label{def-energy}
 The natural energy functional associated to the triple
 $(\rho_k,\u_k,f_k)$    solution to the approximation  system \eqref{approximation-1}-\eqref{regularized initial data} is given by
\begin{equation}\begin{split}\label{energy-functional}
E(t):=E(\rho_k,\u_k,f_k)(t)\ &:=\ \int_{\O}(\frac{1}{2}\rho_k|\u_k|^2+\frac{\rho_k^{\gamma}}{\gamma-1}+\frac{\delta}{\beta-1}\rho_k^{\beta})\,dx\nonumber\\
&+\int_{\O}\int_a^b\int_{\R^3}r^3(1+|\xi|^2)f_k\,d\xi\,dr\,dx,
\end{split}\end{equation}
The corresponding initial energy is
\begin{equation}\label{energy-functional-ini}
E_0:=\int_{\O}(\frac{\m_0^2}{2\rho_0}+\frac{\rho_0^{\gamma}}{\gamma-1}+\frac{\delta}{\beta-1}\rho_0^{\beta})\,dx+\int_{\O}\int_a^b\int_{\R^3}r^3(1+|\xi|^2)f_0\,d\xi\,dr\,dx.
\end{equation} \end{Definition}

The desired estimates will follow from the following result.

  \begin{Proposition}[{\sl The Energy Inequality}]
 \label{Proposition at the first level}
 Let the triple $(\rho_k,\u_k,f_k)$  be the solution to system \eqref{approximation-1}-\eqref{regularized initial data} constructed above, then
for any $T>0$, the   $(\rho_k,\u_k,f_k)$ satisfies the following energy inequality
\begin{equation}
\label{approximation-energy inequality}
\begin{split}
& E(t)
\ +\ \mu\int_0^T\int_{\O}|\nabla\u_k|^2\,dx\,dt+\lambda\int_0^T\int_{\O}|\Dv\u_k|^2\,dx\,dt\\
&\qquad\qquad\qquad   +\ \varepsilon\int_0^T\int_{\O}(\gamma\rho_k^{\gamma-2}+\delta\beta\rho_k^{\beta-2})|\nabla\rho_k|^2\,dx\,dt
\  \leq \ E_0.
\end{split}
\end{equation}
 \end{Proposition}

\begin{proof} First,
taking $\varphi=\u_k$ in \eqref{integral equation-1}, one obtains the following identity corresponding to the regularized Navier-Stokes part \eqref{approximation-NS part}
 \begin{equation}
 \label{energy-NS part}
 \begin{split}
&\frac{d}{dt}\int_{\O}(\frac{1}{2}\rho_k|\u_k|^2+\frac{\rho_k^{\gamma}}{\gamma-1}+\frac{\delta}{\beta-1}\rho_k^{\beta})\,dx
\\&+\mu\int_{\O}|\nabla\u_k|^2\,dx+\lambda\int_{\O}|\Dv\u_k|^2\,dx+\varepsilon\int_{\O}(\gamma\rho_k^{\gamma-2}+\delta\beta\rho_k^{\beta-2})|\nabla\rho_k|^2\,dx
\\&+\int_{\O}\mathfrak{n}_k|\u_k|^2\,dx=\int_{\O}j_k\u_k\,dx,
\end{split}
 \end{equation}
 for any  $t\in [0,T'].$ Next,
 applying Proposition \ref{energy equality for kinetic part}, and adding \eqref{energy-NS part}, we obtain
 the following $L^2$ energy identity for the whole system that includs the kinetic equation \eqref{kinetic}:
\begin{equation*}
\begin{split}
&\frac{d}{dt}\left(\int_{\O}(\frac{1}{2}\rho_k|\u_k|^2+\frac{\rho_k^{\gamma}}{\gamma-1}+\frac{\delta}{\beta-1}\rho_k^{\beta})\,dx+\int_{\O}\int_a^b\int_{\R^3}r^3(1+|\xi|^2)f_k\,d\xi\,dr\,dx\right)
\\&+\mu\int_{\O}|\nabla\u_k|^2\,dx+\lambda\int_{\O}|\Dv\u_k|^2\,dx+\varepsilon\int_{\O}(\gamma\rho_k^{\gamma-2}+\delta\beta\rho_k^{\beta-2})|\nabla\rho_k|^2\,dx
\\&+\int_{\O}\int_a^b\int_{\R^3}r f_k|\u_k-\xi|^2\,d\xi\,dr\,dx=0
\end{split}
\end{equation*}
on $[0,T'].$

Integrating  with respect to $t$, we deduce
 the following energy identity
\begin{equation*}
\label{energy identity for approximation}
\begin{split}& E(\rho_k,\u_k,f_k)(t)
+\mu\int_0^{T_{k}}\int_{\O}|\nabla\u_k|^2\,dx\,dt+\lambda\int_0^{T_{k}}\int_{\O}|\Dv\u_k|^2\,dx\,dt\\
&\qquad\qquad +\ \varepsilon\int_0^{T_{k}}\int_{\O}(\gamma\rho_k^{\gamma-2}+\delta\beta\rho_k^{\beta-2})|\nabla\rho_k|^2\,dx\,dt\\
&\qquad\qquad \qquad+\ \int_0^{T_{k}}\int_{\O}\int_a^b\int_{\R^3}r f_k|\u_k-\xi|^2\,d\xi\,dr\,dx\,dt
\ = \ E_0,
\end{split}
\end{equation*}
on $[0,T'],$ where the total energy energy $E(t)=E(\rho_k,\u_k,f_k)(t)$ and its initial form $E_0$ were defined in
\eqref{energy-functional} and \eqref{energy-functional-ini}, respectively.

In particular, since both terms
\begin{equation*}\begin{split}
&\qquad\ \varepsilon\int_0^{T_{k}}\int_{\O}(\gamma\rho_k^{\gamma-2}+\delta\beta\rho_k^{\beta-2})|\nabla\rho_k|^2\,dx\,dt\\
&\text{and }\\
&\qquad\ \int_0^T\int_{\O}\int_a^b\int_{\R^3}r f_k|\u_k-\xi|^2\,d\xi\,dr\,dx\,dt,
\end{split}\end{equation*}
are non-negative, then  the energy inequality \eqref{approximation-energy inequality} naturally.
\end{proof}

 the   energy inequality \eqref{approximation-energy inequality}, together with estimate \eqref{boundness of density}, yield the following uniform bounds in $k$ and $\varepsilon$, for the the components of the triple solutions to system \eqref{approximation-1}-\eqref{regularized initial data}
\begin{equation}
\begin{split}\label{coro-energy}
&\|\u_k\|_{L^{\infty}(0,T;L^2(\O))}\leq C_0<\infty,\\
&\|\rho_k\|_{L^{\infty}(0,T;L^{\gamma}(\O))}\leq C_0<\infty,
\\&\|\nabla\u_k\|_{L^2(0,T;L^2(\O))}\leq C_0<\infty,
\end{split}
\end{equation}
where $C_0$ only depends on the initial data.

To end, noting  that the  $L^{\infty}(X_k)$ and $L^2(X_k)-$norms are equivalent on the finite dimensional space $X_k$, then
\begin{equation*}\sup_{t\in[0,T_k]}\left(\|\u_k\|_{L^{\infty}}+\|\nabla\u_k\|_{L^{\infty}}\right)\leq C_0(E_0).
\end{equation*}
As a consequence of this observation, the existence time interval $[0,T']$
can be extended to $[0,T].$ for all $T>0.$

Hence, the existence proof  of a weak solution triple $(\rho^{\eps,\delta}_k,\u^{\eps,\delta}_k,f^{\eps,\delta}_k)$ to the regularization \eqref{approximation-1}-\eqref{regularized initial data} for any $T>0$ is completed.

\vskip0.3cm

\section{Recover the weak solutions}

 In order to complete Theorem \ref{T1},we need to recover weak solutions to \eqref{operator}-\eqref{NSVB3}. To this end, we pass to the limits in the following order, as $k\to\infty$, next $\varepsilon\to 0$ and finally $\delta\to 0$,  for the unique solutions constructed as in Proposition \ref{Proposition at the first level}. Here we use the triple $(\rho_k,\u_k,f_k)$ to denote the solution constructed as in Proposition \ref{Proposition at the first level}, were still omit the supraindexed $\eps$ and $\delta$ for notation simplicity.

Thanks to \eqref{approximation-energy inequality},  the following uniformly estimates hold
\begin{equation}
\label{NS-estimate1}
\|\sqrt{\rho_k}\u_k\|_{L^{\infty}(0,T;L^2(\O))}\leq C<\infty,
\end{equation}

\begin{equation}
\label{NS-estimate2}
\|\rho_k\|_{L^{\infty}(0,T;L^{\gamma}(\O))}\leq C<\infty,
\end{equation}

\begin{equation}
\label{NS-estimate3}
\|\nabla\u_k\|_{L^2(0,T;L^2(\O))}\leq C<\infty,
\end{equation}

\begin{equation}
\label{NS-estimate4}
\delta\int_{\O}\frac{1}{\beta-1}\rho_k^{\beta}\,dx\leq C<\infty,\quad\text{ for any } t\in(0,T),
\end{equation}

\begin{equation}
\varepsilon\int_0^T\int_{\O}(\gamma\rho_k^{\gamma-2}+\delta\beta\rho_k^{\beta-2})|\nabla\rho_k|^2\,dx\,dt\leq C<\infty,
\end{equation}

\begin{equation}
\label{VB-estimate1}
\int_{\O}\int_a^b\int_{\R^3}r^3(1+|\xi|^2)f_k\,d\xi\,dr\,dx\leq C<\infty,\quad\text{ for any } t\in(0,T).
\end{equation}\\
\newline
Then a consequence we can show
 the following Lemma.
\begin{Lemma}
\label{Lemma n-j approximation}
There exists a constant $C$ independent on index $k,$ and regularization parameters $\varepsilon$ and $\delta$ such that
\begin{equation}\label{zero moment estimate in approximation}
\|\mathfrak{n}_k(t)\|_{L^{\infty}(0,T;L^2(\O))}\leq C,
\end{equation}
\begin{equation}
\label{first moment estimate in approximation}
\|j_k(t)\|_{L^{\infty}(0,T;L^{\frac{3}{2}}(\O))}\leq C.
\end{equation}
\end{Lemma}
\begin{proof} By \eqref{NS-estimate3}, we have $$\|\u_k\|_{L^2(0,T;L^6(\O))}\leq C,$$
where $C$ is uniform in $k$, $\varepsilon$ and $\delta$; and hence $\u_k$ is also uniformly bounded in $L^2(0,T;L^6(\O)).$
 Therefore, taking $N=p=3$ in Lemma \ref{Lemma of moment} and Lemma \ref{Lemma of n-j}, then \eqref{zero moment estimate in approximation} and \eqref{first moment estimate in approximation} follow.
\end{proof}

\vskip0.3cm

The next step is to  show that the limit in $k$ for the sequence of solution $(\rho_k,\u_k,f_k)$ exists in the following sense.
  \begin{Proposition}
  \label{convergence of ns part}
  Let the solutions of $(\rho_k,\u_k,f_k)$ constructed in Proposition \ref{Proposition at the first level}, then for any $\gamma>\frac{3}{2}$,
$$\rho_k\to\rho\quad\text{ in } L^1((0,T)\times\O)\quad\text{ and } C([0,T]; L^{\gamma}_{weak}(\O)),$$
$$\u_k\to\u\quad\text{ weakly in } L^2(0,T;W_0^{1,2}(\O)),$$
$$\rho_k\u_k\to\rho\u\quad\text{ in } C([0,T];L^{\frac{2\gamma}{\gamma+1}}_{weak}(\O)),$$
and $$\rho_k^{\gamma}\to\rho^{\gamma}\quad\text{ in } L^{\frac{\gamma+\theta}{\gamma}}((0,T)\times\O)\quad\text{ for some }0< \theta<\frac{\gamma}{3}.$$
  \end{Proposition}
\begin{Remark}
The proof of this proposition follows from techniques developed by Lions \cite{L2} and Feireisl \cite{F,FNP,F04} applied to the compressible Navier-Stokes equations with the external forces. They are crucial for the limiting
process of the solution to the whole fluid-kinetic system. In the sake of completeness we write some of these estimates in the actual larger system context.
\end{Remark}

The uniform estimate \eqref{more estimate on density} hold for the solutions of the compressible Navier-Stokes equations, even with the external force, if it is in $L^{p}(0,T;L^q(\O))$ for some $p,q>1.$  For the more detail, we refer the readers to \cite{F,FNP,F04, L2}. Thus, the first step consist in  controlling the uniform estimate of the force term in $k,$ $\delta$ and $\varepsilon$, namely
  \begin{equation}
  \label{external force}
  -\int_a^b\int_{\R^3}r(\u_k-\xi)f_k\,d\xi\,dr=-\mathfrak{n}_k \u_k+j_k,
  \end{equation}
 which has been proved to be bounded in $L^p(0,T;L^q(\O))$ for some $p,q>1$, uniformly in $k$, $\delta$ and $\varepsilon$.
In fact, we have
$$\|j_k-\mathfrak{n}_k\u_k\|_{L^{2}(0,T;L^{\frac{3}{2}}(\O))}\leq C\|j_k\|_{L^{\infty}(0,T;L^{\frac{3}{2}}(\O))}+C\|\mathfrak{n}_k\|_{L^{\infty}(0,T;L^2(\O)}\|\u_k\|_{L^2(0,T;L^6(\O))},$$
and hence $j_k-\mathfrak{n}_k\u_k$ is uniformly bounded in $L^2(0,T;L^{\frac{3}{2}}(\O))$.

Note that $
  -\int_a^b\int_{\R^3}r(\u_k-\xi)f_k\,d\xi\,dr$ is bounded in $L^2(0,T;L^{\frac{3}{2}}(\O))$, we can apply the argument in \cite{F,FNP,F04, L2} to  \eqref{approximation-1}. We obtain the following estimate
in Lemma \ref{Lemma on density}.
\begin{Lemma}
\label{Lemma on density}
For any $\gamma>\frac{3}{2}$, there exists a constant $0<\theta<\frac{\gamma}{3}$, depending on $\gamma$, such that
\begin{equation}
\label{more estimate on density}
\int_0^T\int_{\O}(a\rho_k^{\gamma+\theta}+\delta\rho_k^{\beta+\theta})\,dx\,dt\leq C<\infty,
\end{equation}
where $C>0$ is uniformly on $n$, $\varepsilon$ and $\delta$.
\end{Lemma}

With above convergence of Proposition \ref{convergence of ns part} in hand, we are ready to pass to the limits for the Navier-Stokes part as $k\to\infty.$ We could use the similar arguments to handle the other limits with respects to $\varepsilon$ and $\delta$.
For more details on the weak stability of the compressible Navier-Stokes equations, we refer the readers to \cite{L,FNP,F04}.

\vskip0.3cm
Headlines focus on the stability of weak solutions to the kinetic equation \eqref{kinetic}.
By \eqref{estimate of f}, we have
\begin{equation}
\label{weak convergence of f}
f_k\rightharpoonup f\quad\; L^{\infty}(0,T;L^p(\O\times\R^3\times\R^+))-\text{ weak$^*$}
\end{equation}
for any $1< p\leq \infty.$

Letting $\varphi(x)$ be a smooth compactly supported test function, we have
\begin{equation}
\label{control jn}
\begin{split}
&\int(j_k-\int_a^b\int_{\R^3}r\xi f\,d\xi\,dr)\varphi(x)\,dx
\\&\leq \int\int\int r(f_k-f)(1+|\xi|)\varphi(x)\,d\xi\,dr\,dx
\\&=\int\int\int \left(r^{\frac{2}{3}}(f_k-f)^{\frac{2}{3}}(1+|\xi|)^{\frac{4}{3}}\varphi^{\frac{2}{3}}(x)\right)\left(r^{\frac{1}{3}}(f_k-f)^{\frac{1}{3}}(1+|\xi|)^{\frac{-1}{3}}\varphi^{\frac{1}{3}}(x)\right)\,d\xi\,dr\,dx
\\&\leq 2 \left(\int\int\int r(f_k-f)(1+|\xi|^2)\varphi(x)\,d\xi\,dr\,dx\right)^{\frac{2}{3}}\left(\int\int\int r(f_k-f)\frac{\varphi(x)}{1+|\xi|}\,d\xi\,dr\,dx\right)^{\frac{1}{3}}
\\&=2C\left(\int\int\int r(f_k-f)\frac{\varphi(x)}{1+|\xi|}\,d\xi\,dr\,dx\right)^{\frac{1}{3}},
\end{split}
\end{equation}
where we used \eqref{VB-estimate1} and a fact
\begin{equation*}
\begin{split}&\left(\int\int\int r(f_k-f)(1+|\xi|^2)\varphi(x)\,d\xi\,dr\,dx\right)^{\frac{2}{3}}
\\&\leq
\left(2\int\int \int r(1+|\xi|^2)f_k\,d\xi\,dr\,dx\right)^{\frac{2}{3}}\leq C.
\end{split}
\end{equation*}
 Thus, the last term in \eqref{control jn} converges to zero
as $k$ goes to infinity since $f_k$ converges to $f$ weakly in $L^2(0,T;L^2(\O\times\R^3\times\R^+))$
 and
$$\frac{r\varphi(x)}{1+|\xi|}\in L_{loc}^2(\O\times\R^3\times\R^+)).$$
It follows that
\begin{equation}
\label{convergence of jk}
j_k\rightharpoonup j\quad \text{weakly in }L^{\infty}(0,T;L^p(\O))
\end{equation}
for any $1<p\leq\frac{3}{2},$
where $j=\int\int r\xi f\,d\xi\,dr$. \\
Similarly, we have that \begin{equation}
\label{nk weak convergence}
\mathfrak{n}_k=\int \int r f_k\,d\xi\,dr \rightharpoonup \mathfrak{n}=\int\int rf\,d\xi\,dr \text{ weakly in } L^2(0,T;L^2_{loc}(\O)).
\end{equation}

By \eqref{estimate of f} again, $f_k$ is uniformly bounded in $L^{\infty}(0,T;L^{\infty}(\O\times\R^3\times \R^{+}).$ Relying on this, we can show the following uniform bounds. With \eqref{nk weak convergence}, we have the weak convergence of $Q(f_k).$
\begin{Lemma}
\label{kernel estimate} If \eqref{estimate of f}, then
$Q(f_k)$ is uniformly bounded in $$L^{\infty}(0,T;L^{\infty}(\O\times\R^3\times\R^{+})\cap L^{\infty}(0,T;L^{p}(\O\times\R^3\times\R^{+})$$ for any $p\geq 1,$ and
\begin{equation}
\label{operator convergence}
\int_a^b\int_{\R^3}Q(f_k)\,d\xi\,dr\rightharpoonup \int_a^b\int_{\R^3}Q(f)\,d\xi\,dr\text{ weakly in } L^2(0,T;L^2(\O)).
\end{equation}
\end{Lemma}

\begin{proof}
\begin{equation*}
\begin{split}
\|Q(f_k)\|_{L^{\infty}}&\leq \nu \|f_k(x,\xi,r,t)\|_{L^{\infty}}+\nu \|f_k(x,\xi,r,t)\|_{L^{\infty}}\int_{r>r^*}B(r^*,r)\,d r^*
\\&\leq (\nu+C\nu) \|f_k(x,\xi,r,t)\|_{L^{\infty}},
\end{split}
\end{equation*}
where we used a fact $$ \int_{r>r^*}B(r^*,r)\,d r^*\leq C.$$
Similarly, \begin{equation*}
\begin{split}
\|Q(f_k)\|_{L^{1}}&\leq \nu \|f_k(x,\xi,r,t)\|_{L^{1}}+\nu\| \int_{r>r^*}B(r^*,r)f_k(t,x,\xi,r^*)\,d r^*\|_{L^1}
\\&\leq \nu \|f_k(x,\xi,r,t)\|_{L^{1}}+ \nu \|f_k(x,\xi,r,t)\|_{L^{1}}\int_{r>r^*}B(r,r^*)\,d r^*
\\&\leq (\nu +c\nu)\|f_k(x,\xi,r,t)\|_{L^{1}}.
\end{split}
\end{equation*}
For any smooth $\varphi(x)$,
\begin{equation*}
\begin{split}&
\int_{\O}\left(\int_a^b\int_{\R^3}Q(f_k)(x,\xi,r,t)\,d\xi\,dr-\int_a^b\int_{\R^3}Q(f_k)(x,\xi,r,t)\,d\xi\,dr\right)\varphi(x)\,dx
\\&\leq\frac{\nu}{a}
\int_{\O}\int_a^b\int_{\R^3}r(f_k-f)\,d\xi\,dr\varphi(x)\,dx+\frac{\nu}{a}
\int_{\O}\int_a^b\int_{\R^3}\int_{r>r^*}rB(r^*,r) (f_k-f)\,d r^*\,d\xi\,dr\,dx
\\&\leq \frac{C \nu}{a}\int_{\O}\int_a^b\int_{\R^3}r(f_k-f)\,d\xi\,dr\varphi(x)\,dx \to 0
\end{split}
\end{equation*}
as $k\to\infty.$
By \eqref{nk weak convergence}, we have \eqref{operator convergence}.
\end{proof}

The last task is to handle the convergence of the right-hand side of \eqref{integral equation-1}$$\int_a^b\int_{\R^3}r \u_k f_k\,d\xi\,d r.$$
To prove this one, we follow the same argument as in \cite{MV}. In fact, we shall use the following lemma, which was from \cite{L2}.
\begin{Lemma}
Let $g^n$ and $h^n$ converge weakly to $g$ and $h$ respectively in $L^{p_1}(0,T;L^{p_2}(\O))$ and $L^{q_1}(0,T;L^{q_2}(\O))$ where $1\leq p_1, q_1\leq +\infty$,
$$\frac{1}{p_1}+\frac{1}{q_1}=\frac{1}{p_2}+\frac{1}{q_2}=1.$$
We assume in addition that
$$\frac{\partial g^n}{\partial t} \text{ is bounded in } L^1(0,T;W^{-m,1}(\O))\text{ for some }m\geq 0 \text{independent of } n $$
and $$ \|h^n-h^n(\cdot+\xi,t)\|_{L^{q_1}(0,T;L^{q_2}(\O))}\to 0 \text{as } |\xi|\to 0, \text{ uniformly in } n.$$
Then, $g^n h^n$ converges to $gh$ in the sense of distributions on $\O\times(0,T).$
\end{Lemma}

 Indeed, we have
$$(\mathfrak{n}_k)_t=-\Dv_x(j_k),$$
and so $(\mathfrak{n}_k)_t$ is bounded in $L^{\infty}(0,T;W^{-1,1}(\O)).$ Since $\nabla\u_k$ is bounded in $L^2(0,T;L^2(\O)),$ we can apply a classical compactness lemma \cite{L2} to have
\begin{equation}
\label{dis weak convergence}
\mathfrak{n}_k\u_k\to \mathfrak{n}\u\;\;\text{in the sense of distributions}.
\end{equation}

 Similarly, we are able to show, as $k\to\infty$,
 \begin{equation}
 \label{convergence of external force}
 \int_{\O}\int_a^b\int_{\R^3}\frac{\u_k-\xi}{r^2}f_k\phi\,d\xi\,d r\, d x\to \int_{\O}\int_a^b\int_{\R^3}\frac{\u-\xi}{r^2}f\phi\,d\xi\,d r\, d x
 \end{equation}
 for any $\phi \in C^1([0,T]\times\O)$ with compact support with respect to $x$.

With Proposition \ref{convergence of ns part}, \eqref{convergence of jk}, \eqref{operator convergence}, \eqref{dis weak convergence} and  \eqref{convergence of external force},  we are ready to pass to the limits in the weak formulation of the Navier-Stokes and in the weak formulation of kinetic equation.
Thus, we are
allowed to pass to the limits as $k$ goes to infinity in the approximation of \eqref{integral equation-1} for the following weak formulations
\begin{equation*}
\begin{split}
\label{integral equation}
&\int_{\O}\rho_k\u_k(t)\cdot\varphi\,dx-\int_{\O}\m_0\cdot\varphi\,dx
=\int_0^t\int_{\O}\left(\mu\D\u_k+\lambda\nabla\Dv\u_k\right)\varphi\,dx\,dt
\\&+\int_0^t\int_{\O}\left(\varepsilon\nabla\u_k\cdot\nabla\rho_k-\Dv(\rho_k\u_k\otimes \u_k)-\nabla \rho_k^{\gamma}-\delta\nabla\rho_k^{\beta}-\mathfrak{n}_k\u_k+j_k\right)\varphi\,dx\,dt,
\end{split}
\end{equation*}
and
\begin{equation*}
\label{2.2+}
\begin{split}
&\quad-\int_{0}^{t}\!\!\!\int_a^b\int_{\O}\int_{\R^3}f_k\left({\partial_t\phi+\xi\cdot\nabla_{x}\phi+\frac{(\u_k-\xi)}{r^2} \cdot\nabla_{\xi}\phi}\right)\;dxd\xi\,d r ds
\\&\quad\quad\quad\quad\quad=\int_a^b\int_{\O}\int_{\R^3}f_{0}\phi(0,\cdot,\cdot)\;dxd\xi\,dr+\int_0^t\int_a^b\int_{\O}\int_{\R^3}Q(f_k) \phi\,d\xi\,dx\,dr\,dt.
\end{split}
\end{equation*} Here we should remark that the all uniform bounds in this section are independent on $\varepsilon$ and $\delta$. Thus, we can pass into the limits as $k\to\infty$, $\varepsilon\to 0$ and $\delta\to 0$ at the same time.
 Thus, all convergence results in this section allow us to recover the weak formulations \eqref{weak-1}-\eqref{weak-2} by passing into the limits as $k\to \infty$, $\varepsilon\to0$
 and $\delta\to 0.$

At last,
 passing to the limits in \eqref{approximation-energy inequality} with respects to $k\to \infty$, $\varepsilon\to0$
 and $\delta\to 0$, the following energy inequality could be obtained in the following Lemma:
\begin{Lemma}
\label{Lemma of energy inequality} If $(\rho,\u)$ is the weak limit of $(\rho_k,\u_k)$ as $k$ goes to infinity, then
 \begin{equation}
\label{energy final}
\begin{split}
&\int_{\O}(\frac{1}{2}\rho|\u|^2+\frac{\rho^{\gamma}}{\gamma-1})\,dx+\int_{\O}\int_a^b\int_{\R^3}r^3(1+|\xi|^2)f\,d\xi\,dr\,dx
\\&+\mu\int_0^T\int_{\O}|\nabla\u|^2\,dx\,dt+\lambda\int_0^T\int_{\O}|\Dv\u|^2\,dx\,dt
\\&\leq \int_{\O}(\frac{\m_0^2}{2\rho_0}+\frac{\rho_0^{\gamma}}{\gamma-1})\,dx+\int_{\O}\int_a^b\int_{\R^3}r^3(1+|\xi|^2)f_0\,d\xi\,dr\,dx.
\end{split}
\end{equation}

In addition, the same  conclusion holds true as the limits  $\varepsilon\to 0$ and $\delta\to 0$.
\end{Lemma}

\begin{proof}
Using the weak convergence  and  the convexity of the
energy,
estimates  \eqref{energy final} follow by passing to the limit from \eqref{approximation-energy inequality} with respect to $k\to \infty$.

Finally, because all estimates for  are also unifomm for both $\varepsilon$ and $\delta$ small, then  the corresponding limiting problem, as both parameters tend to zero, yield a solution to the problem posed in Theorem \ref{T1}.
\end{proof}
Thus, we have completed the proof of our main result Theorem \ref{T1}.
\vskip0.3cm
\bigskip

\bigskip\bigskip


\section*{Acknowledgments}
 I. M. Gamba acknowledges support from NSF grant DMS-1413064 and  DMS-1715515. C. Yu  acknowledges support from NSF grant DMS-1540162. This work was also partially funded by NSF  RNMS-1107465.
 The authors thank the support from the Institute of Computational Engineering and Sciences (ICES) at the University of Texas Austin.

\end{document}